\documentclass[12pt,nosumlimits,nonamelimits]{amsart}
\usepackage{amssymb,mathtools,accents}

\newtheorem{thm}{Theorem}
\newtheorem{prop}[thm]{Proposition}
\newtheorem{lem}[thm]{Lemma}
\theoremstyle{remark}

\theoremstyle{definition}
\newtheorem*{defn}{Definition}
\newcommand{\IB}{\mathbb{B}}
\newcommand{\IC}{\mathbb{C}}
\newcommand{\IK}{\mathbb{K}}

\newcommand{\IN}{\mathbb{N}}
\newcommand{\IP}{\mathbb{P}}
\newcommand{\IR}{\mathbb{R}}
\newcommand{\IT}{\mathbb{T}}

\newcommand{\bF}{\mathbf{F}}

\newcommand{\cC}{\mathcal{C}}

\newcommand{\cH}{\mathcal{H}}

\newcommand{\cO}{\mathcal{O}}

\newcommand{\cT}{\mathcal{T}}
\newcommand{\cU}{\mathcal{U}}
\newcommand{\cZ}{\mathcal{Z}}

\newcommand{\ve}{\varepsilon}
\newcommand{\vp}{\varphi}
\newcommand{\Ga}{\Gamma}
\newcommand{\La}{\Lambda}
\newcommand{\Tree}{\mathrm{T}}

\newcommand{\acts}{\curvearrowright}
\newcommand{\id}{\mathrm{id}}

\DeclareMathOperator{\rank}{rank}
\DeclareMathOperator{\supp}{supp}
\newcommand{\cspan}{\mathop{\overline{\mathrm{span}}}}
\newcommand{\ip}[1]{\mathopen{\langle}#1\mathclose{\rangle}}
\def\bigast{\font\bigsymbolsfont=cmsy10 scaled \magstep3
 \setbox0=\hbox{\bigsymbolsfont\char'003 }\mathord{\lower1pt\box0}}\relax\ignorespaces
\title{Proximality and selflessness for group C*-alge\-bras}
\author{Narutaka Ozawa}
\address{RIMS, Kyoto University, \mbox{606-8502} Japan}
\email{narutaka@kurims.kyoto-u.ac.jp}
\thanks{The author was partially supported by JSPS KAKENHI Grant Numbers 24K00527, 25H00588, 25H00593}

\date{\today}

\begin{document}
\begin{abstract}
We prove that the reduced group $\mathrm{C}^*$-alge\-bras 
of infinite countable discrete groups having topologically-free 
extreme boundaries, or more generally groups that satisfy 
certain combinatorial property including 
all acylindrically hyperbolic groups with no nontrivial finite normal subgroups 
and all Zariski-dense subgroups of $\mathrm{PSL}(n,\mathbb{R})$, 
are selfless in the sense of L. Robert. 
This generalizes the recent results of 
Amrutam, Gao, Kunnawalkam Elayavalli, and Patchell, and of Vigdorovich.
We also prove that selflessness is stable under 
tensor product among exact $\mathrm{C}^*$-alge\-bras 
and that a $\mathrm{C}^*$-prob\-a\-bility space is selfless 
provided that it is either simple and purely infinite or simple, 
exact, $\mathcal{Z}$-stable, and uniquely tracial. 
\end{abstract}
\maketitle
\section{Introduction}
A remarkable property called \emph{selflessness} is recently 
introduced for a $\mathrm{C}^*$-alge\-bra by L. Robert (\cite{robert}) 
and it quickly attracted a number of researchers' attention 
(see, e.g., \cite{agkep}, \cite{hker}, \cite{rtv}, \cite{vigdorovich})
as it implies many important regularity properties such as 
simplicity, (in the tracial setting) stable rank one, 
and strict comparison. 
In particular,  Amrutam, Gao, Kunnawalkam Elayavalli, 
and Patchell (\cite{agkep}) has proved that a large family 
of groups, namely acylindrically hyperbolic groups 
with the \emph{rapid decay property} and with no nontrivial 
finite normal subgroups are \emph{$\mathrm{C}^*$-selfless}. 
Here, we say a group $\Ga$ is $\mathrm{C}^*$-selfless if it 
its reduced group $\mathrm{C}^*$-alge\-bra 
$\mathrm{C}^*_\lambda(\Ga)$ is selfless. 

In this paper, we prove that every infinite countable 
discrete group with a minimal topologically-free 
extremely-proximal action is $\mathrm{C}^*$-selfless. 
Our method is topological, as opposed to the more 
analytical one in \cite{agkep}, and to construct 
a kind of ``tree-graded'' space (\cite{ds}) out of an extreme 
boundary (see Section~\ref{sec:space}). 
Note that every $\mathrm{C}^*$-selfless group 
is $\mathrm{C}^*$-simple (\cite{robert}) and 
whether the converse is also true is 
an open problem (Problem XCI in \cite{stw}).
We remind Kalantar and Kennedy's theorem (\cite{kk}) that 
a discrete group is $\mathrm{C}^*$-simple if and only if 
it has a minimal topological-free strongly-proximal action. 
As the names suggest, strong proximality is weaker than 
extreme proximality (\cite{glasner}). 

\begin{defn}
Let $\Ga$ be a countable discrete group and $\Ga\acts X$ 
be an action on a compact topological space $X$. 
The action $\Ga\acts X$ (or the $\Ga$-space $X$) 
is called an \emph{extreme boundary} 
if it is minimal and \emph{extremely proximal} (\cite{glasner}) 
in the sense that for every non-empty 
open subsets $U$ and $V$ of $X$, there is 
$g\in\Ga$ such that $g(X\setminus U)\subset V$.

For $x\in X$, we denote by $\Ga_x \coloneq \{ g\in\Ga : gx=x \}$ 
the stabilizer subgroup at $x$. 
The action $\Ga\acts X$ is said to be 
\emph{topologically-free} if the set of points 
with trivial stabilizer groups is dense in $X$. 
Note that if an extreme boundary $\Ga\acts X$ 
is not topologically-free, then every proper 
closed subset of $X$ is pointwise fixed by some 
$g\in\Ga\setminus\{1\}$. 

We say a sequence $(z_n)_n$ in $\Ga$ is \emph{axial} if 
there is a topologically-free extreme boundary 
$\Ga\acts X$ with distinct points $z_\pm\in X$ 
that satisfies the following two conditions.
(1):~For every neighborhoods $U_\pm$ of $z_\pm$ 
one has 
$z_n (X\setminus U_-) \subset U_+$
eventually, or equivalently that 
$z_n^{-1}( X \setminus U_+ ) \subset U_-$ eventually. 
(2):~The $\Ga$-action on $\{z_\pm\}$ is free 
in the sense that $g\{z_\pm\}\cap\{z_\pm\}\neq\emptyset$ 
implies $g=1$. 
Note that a second countable extreme boundary $X$ 
with $|X|>2$ accommodates an axial sequence if and only if 
it is topologically-free. 
\end{defn}

\begin{thm}\label{thm:main}
An infinite countable discrete group $\Ga$ having 
a topologically-free extreme boundary is 
$\mathrm{C}^*$-selfless. 
More precisely, for any axial sequence $(z_n)_n$ in $\Ga$ 
and any free ultrafilter $\cU$ on $\IN$, the homomorphism 
\[
\Ga*\ip{z}\to\mathrm{C}^*_\lambda(\Ga)^\cU,
\]
given by the diagonal embedding of 
$\Ga$ into $\mathrm{C}^*_\lambda(\Ga)^\cU$ 
and $z\mapsto [z_n]_n \in \mathrm{C}^*_\lambda(\Ga)^\cU$, 
induces a faithful embedding of the reduced 
group $\mathrm{C}^*$-alge\-bra $\mathrm{C}^*_\lambda(\Ga*\ip{z})$ 
into $\mathrm{C}^*_\lambda(\Ga)^\cU$.
\end{thm}

See, e.g., \cite{bowditch}, \cite{bio}, \cite{flmms}, \cite{io}, 
\cite{jr}, \cite{ls}, \cite{lbmb} for examples of extreme 
boundaries. Every group satisfying the conclusion 
of Theorem~\ref{thm:main} is \emph{mixed-identity-free} 
(see \cite{ho} for the terminology) 
and converse holds true if $\Ga$ has a faithful minimal action 
of general type on a tree (\cite{flmms}), since 
such an action gives rise to an extreme boundary. 
In particular, non-elementary free product groups 
are $\mathrm{C}^*$-selfless. 
For the case of amalgamated free products and 
HNN extensions, see \cite{flmms} and \cite{io} for 
a characterization of topological freeness of 
the corresponding Bass--Serre tree compactifications. 
Non-elementarily relatively hyperbolic 
groups with no nontrivial finite normal 
subgroups are $\mathrm{C}^*$-selfless, as their Bowditch 
compactifications (\cite{bowditch}) are topologically-free 
extreme boundaries. 
It is not clear whether this generalizes to 
acylindrically hyperbolic groups (with no assumption 
on the rapid decay property, cf.\ \cite{agkep}), because their hyperbolic 
structure need not yield suitable compactifications. 
We fix this problem in Section~\ref{sec:pow} and prove 
that a group satisfying property $\mathrm{P}_{\mathrm{PHP}}$ 
(defined there) is $\mathrm{C}^*$-selfless. The class of groups 
with property $\mathrm{P}_{\mathrm{PHP}}$ contains 
all acylindrically hyperbolic groups with no nontrivial 
finite normal subgroups as well as all Zariski-dense subgroups 
of $\mathrm{PSL}(d\geq2,\IR)$ (cf.\ \cite{vigdorovich}). 

There are $\mathrm{C}^*$-selfless groups outside 
the mixed-identity-free realm. 
\begin{thm}\label{thm:tensor}
Let $(A_i,\vp_i)$ be separable $\mathrm{C}^*$-prob\-a\-bility spaces.
Assume that all $(A_i,\vp_i)$ are selfless and exact. 
Then the tensor product $\bigotimes_i (A_i,\vp_i)$ is selfless. 
In particular, the class of countable discrete groups that are 
$\mathrm{C}^*$-selfless and exact is closed under direct product. 
\end{thm}

We introduce the notion of complete selflessness in Section~\ref{sec:cs} and 
prove that the examples in Theorem~\ref{thm:main} and their tensor product 
are in fact completely selfless, while 
no nuclear tracial $\mathrm{C}^*$-prob\-a\-bility space is completely selfless. 
The following is a partial converse to 
Theorem 3.1 in \cite{robert} and solves Question 5.4 in \cite{robert}. 
Note that by Matui and Sato's theorem (\cite{ms}), 
$\cZ$-stability is necessary for a nuclear 
tracial $\mathrm{C}^*$-prob\-a\-bility space to be selfless.

\begin{thm}\label{thm:general}
A simple and purely infinite 
$\mathrm{C}^*$-prob\-a\-bility space $(A,\vp)$ 
is completely selfless. 
A simple, exact, $\cZ$-stable, and uniquely tracial 
$\mathrm{C}^*$-prob\-a\-bility space $(A,\tau)$ 
is selfless. 
\end{thm}

\subsection*{Notations}
We use the symbol ``$\pm$'' for a slightly abusive notation 
for ``$+$ and/or $-$'' or ``$+1$ and/or $-1$''. 
We denote by $A^\cU$ the ultrapower of 
a $\mathrm{C}^*$-alge\-bra $A$ and 
by $[a_n]_n$ the element in $A^\cU$ 
represented by a sequence $(a_n)_n$ in $A$. 
A \emph{$\mathrm{C}^*$-prob\-a\-bility space}   
is a pair $(A,\vp)$ of a unital $\mathrm{C}^*$-alge\-bra $A$ 
and a distinguished state $\vp$ which is not necessarily faithful 
but assumed GNS-faithful (a.k.a.\ non-degenerate) 
unless the $\mathrm{C}^*$-alge\-bra 
in consideration is an ultrapower or an ultraproduct. 

\subsection*{Acknowledgments}
A part of this work was carried out while the author was visiting 
the Mathematisches Forschungsinstitut Oberwolfach for 
the workshop ``C*-Algebras'', August 3--8, 2025 and 
the Isaac Newton Institute for the program 
``Operators, Graphs, Groups'' in October 2025. 
He acknowledges the kind hospitality and the exciting
environment provided by the institutes. 
He thanks G. Patchell for the stimulating talk at MFO 
and a useful comment on this paper, 
E.~Breuillard and I.~Vigdorovich
for insights into $\mathrm{PSL}(d,\IR)$, 
and R. Arimoto, L.~Robert, and T. Takeishi for helpful comments. 
\section{Trees and covariant representations}\label{sec:repn}
We recall the notion of trees and their compactification. 
A tree is a connected graph without nontrivial cycles. 
See Section 5.2 and Appendix E in \cite{bo} for more 
on this theme. 
Here we consider a countable tree $\Tree$ (of infinite degree).
The tree $\Tree$ is identified with its vertex set. 
A \emph{geodesic ray} is an infinite sequence 
$\omega(0),\omega(1),\ldots$ in $\Tree$ 
such that $\omega(n)$ adjacent to $\omega(n-1)$ 
and $\omega(n)\neq \omega(n-2)$. 
Geodesic rays $\omega$ and $\omega'$ are \emph{equivalent} 
if there are $N$ and $N'$ such that 
$\omega(N+n)=\omega'(N'+n)$ for all $n\geq0$.
The \emph{boundary} $\partial\Tree$ is the equivalence classes 
of geodesic rays and 
$\bar{\Tree} \coloneq \Tree \cup \partial\Tree$ 
is the compactification of $\Tree$ (whose topology 
will be introduced shortly). 
We define $f\colon\Tree\times\bar{\Tree}\to\Tree$ as follows. 
For $s\in\Tree$, we set $f(s,s)=s$. 
For distinct $s,t\in\Tree$, we define $f(t,s)\in\Tree$ 
to be the unique point between $t$ and $s$ 
that is adjacent to $t$. 
For $t\in\Tree$ and $\omega\in\partial\Tree$, 
we define $f(t,\omega) \coloneq f(t,\omega(n))$ for any $n$ large enough. 
For each $t\in\Tree$, we consider a copy $\Tree_t$ 
of $\Tree$ (or the union of $t$ and its adjacent points) 
and equip $\Tree_t$ a compact topology by taking 
the one-point compactification $\Tree\cup\{\infty\}$ 
of $\Tree$ and identify the new point $\infty$ with $t$. 
We consider the embedding 
\[
\bar{\Tree} \ni \omega \mapsto 
(f(t,\omega))_t \in \prod_{t\in\Tree}\Tree_t
\]
and equip $\bar{\Tree}$ with the induced topology. 
This makes $\bar{\Tree}$ a second-countable 
compact topological space such that every automorphism 
on $\Tree$ extends to a homeomorphism on $\bar{\Tree}$ 
(Proposition 5.2.5 in \cite{bo}). 
Note that if $s\in\Tree$ and $t_n\in\Tree$ are such that 
$f(s,t_n)\to\infty$ in $\Tree$, then $t_n\to s$ in $\bar{\Tree}$. 

\begin{thm}\label{thm:repn}
Let $\La$ be a countable discrete group acting 
on a countable tree $\Tree$ and 
on a $\mathrm{C}^*$-alge\-bra $C$. 
Let $\pi\colon \La\ltimes_{\max} C\to\IB(\cH)$ be a 
covariant representation 
and assume that there is a $\La$-equivariant 
representation of $C(\bar{\Tree})$ into the commutant 
$\pi(C)'$ of $\pi(C)$. 
Then $\pi$ is continuous on the reduced 
crossed product $\La\ltimes_{\min} C$ 
if for every $[t]\in\Tree/\La$ 
the restriction of $\pi$ to the stabilizer subgroup 
$\La_t \coloneq \{ g\in\La : gt=t\}$ is continuous on 
the reduced crossed product $\La_t\ltimes_{\min} C$. 
\end{thm}
\begin{proof}
Let $\pi\colon \La\ltimes_{\max} C\to\IB(\cH)$ be given. 
We denote by $\tau$ the unitary representation 
of $\La$ on $\ell_2\Tree$ induced by $\La\acts\Tree$
and denote by $\sigma$ the $\Lambda$ action on $G$.  
We claim that $\pi$ is weakly contained in the representation 
$\tilde{\pi}$ on 
$\ell_2\Tree \otimes\cH$ given by 
$g\mapsto (\tau\otimes\pi)(g)$ for $g\in\La$ 
and 
$a\mapsto 1\otimes \pi(a)$ for $a\in C$. 
Indeed, there is a sequence  
$\zeta_n\colon\bar{\Tree}\to\mathrm{Prob}(\Tree)$ 
of Borel maps such that 
$\sup_{x\in\bar{\Tree}}\| g\zeta_n^x - \zeta_n^{gx} \|_1\to0$
for every $g\in\La$ (Lemma 5.2.6 in \cite{bo}). 
It follows that $h_n^t(x) \coloneq \zeta_n^x(t)$ defines sequences 
of positive Borel functions on $\bar{\Tree}$ such that 
$\sum_{t\in\Tree} h_n^t=1$ and 
\[
\| \sum_{t\in\Tree} | h_n^{gt} - \sigma_g(h_n^t) | \|_\infty \to 0.
\]
NB: the above sum is an infinite sum of positive Borel functions. 
The $\La$-equi\-variant representation 
of $C(\bar{\Tree})$ into $\pi(C)'$ uniquely 
extends to a \emph{$\sigma$-normal} $\La$-equi\-variant 
representation $\theta$ from 
the $\mathrm{C}^*$-alge\-bra $B(\bar{\Tree})$ 
of bounded Borel functions into $\pi(C)'$. 
By ``$\sigma$-normal'', 
we mean that it sends a bounded pointwise convergent 
sequence to a strong operator topology convergent sequence. 
We define isometries 
$V_n\colon \cH\to\ell_2\Tree\otimes\cH$ 
by $V_n\xi \coloneq \sum_{t\in\Tree} \delta_t \otimes \theta(h_n^t)^{1/2}\xi$, 
or equivalently, $V_n$ is the column operator $(\theta(h_n^t)^{1/2})_{t\in\Tree}$.
Then 
\[
\|V_n \pi(g) - (\tau\otimes\pi)(g)V_n\|^2 
=\| \sum_t |\theta(h_n^{gt})^{1/2}-\theta(\sigma_g(h_n^t))^{1/2}|^2 \|_\infty \to 0
\]
for every $g\in\La$ 
and $V_n\pi(a) = (1\otimes\pi(a))V_n$ for $a\in C$. 
This proves the claim that $\pi$ is weakly contained in $\tilde{\pi}$. 
Thus continuity of $\pi$ follows from that of $\tilde{\pi}$. 

For the continuity of $\tilde{\pi}$, it suffices to consider the continuity 
of the vector states $\vp$ associated with the vectors of 
the form $\delta_t\otimes\xi$, since they form a cyclic family. 
But $\vp = \psi\circ E$, where $E$ is the canonical 
conditional expectation 
from $\La\ltimes_{\max}C$ onto $\La_t \ltimes_{\max} C$ 
and $\psi$ is the vector state associated with $\xi$. 
(We do not need the fact that the natural embedding 
$\La_t \ltimes_{\max} C\subset \La\ltimes_{\max}C$ 
is faithful.)
Hence the continuity on the reduced crossed product 
$\La\ltimes_{\min}C$ follows from that 
for $\La_t \ltimes_{\min}C$ for $t\in\Tree$ 
one from each $\La$-orbit. 
\end{proof}

\section{Extreme boundaries and ``tree-graded'' spaces}\label{sec:space}
Let $\tilde{\Ga} \coloneq \Ga * \ip{z}$. 
By the \emph{reduced form} of $g\in\tilde{\Ga}$, 
we mean the unique expression 
\[
g = a_1z^{\epsilon_1} a_2 z^{\epsilon_2} \cdots a_l z^{\epsilon_l} a_{l+1}
\]
with $a_i\in\Ga$ and $\epsilon_i\in\{\pm\}$ such that 
$i>1$ and $a_i=1$ implies $\epsilon_{i-1} = \epsilon_i$. 
We introduce the tree structure 
on $\Tree \coloneq \tilde{\Ga}/\Ga$ 
on which $\tilde{\Ga}$ acts. 
We declare that the points $s,t\in\Tree$ are adjacent if 
$s^{-1}t \in\{ \Ga z^\pm\Ga\}$. 
With this graph structure $\Tree$ 
becomes a tree (of infinite degree). 
We call the point $o \coloneq \Ga\in\Tree$ the \emph{origin}. 
Every $t\in\Tree$ has unique representative 
$\rho(t)\in\tilde{\Ga}$ whose reduced form ends by $z^\pm$, 
or $\rho(o)=1$ if $t=o$. 
Note that $\rho(at)=a\rho(t)$ 
for all $a\in\Ga$ and $t\in\Tree\setminus\{o\}$ 
and that $\rho(z^\pm t)=z^\pm\rho(t)$ for all $t\in\Tree$. 

Suppose first that $s$ and $t$ are adjacent. 
We label the oriented edge from $t$ to $s$ by 
$\ell(t,s) \coloneq \kappa(\rho(t)^{-1}s)\in \Ga\{ z^\pm \}$, 
where $\kappa(az^\pm \Ga) \coloneq az^{\pm}$.  
Note that $t$ and $\ell(t,s)$ together uniquely determines $s$. 
E.g., $\ell(az\Ga,azbz\Ga)=bz$ 
and $\ell(azbz\Ga,az\Ga)=z^{-1}$ for every $a,b\in\Ga$. 
The Figure~\ref{fig:1} shows how $\Tree$ and $\ell$ look like. 
Next we extend the \emph{label map} on 
$(\Tree\times\bar{\Tree})\setminus\{\mbox{diagonal}\}$
by setting $\ell(t,s)\coloneq\ell(t,t_1)$, 
where $t_1$ is the unique point between 
$t$ and $s$ that is adjacent to $t$, or colloquially, $\ell(t,s)$ 
is the first $\Ga\{z^\pm\}$ component in the reduced 
form of $\rho(t)^{-1}\rho(s)$.  
For example, $\ell(az,azbzcz\Ga)=bz$ and 
$\ell(cz\Ga,azbz\Ga)=z^{-1}$ unless $a=c$. 
\begin{figure}
\includegraphics[height=4cm]{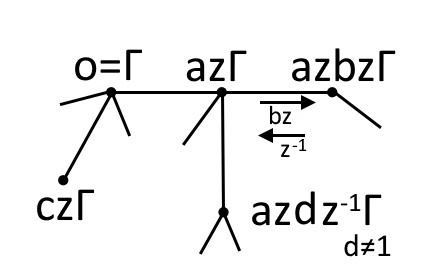}
\caption{The tree $\Tree$}\label{fig:1}
\end{figure}
The following is not hard to see. 

\begin{lem}\label{lem:labelequiv}
Let $(t,s)\in (\Tree\times\bar{\Tree})$ 
be a non-diagonal point. 
For $a\in\Ga$, one has $\ell(o,as)=a\ell(o,s)$ 
and $\ell(at,as) = \ell(t,s)$ for $t\neq o$. 
Also, $\ell(zt,zs)=\ell(t,s)$. 
\end{lem}

Let $\Ga\acts X$ be an extreme boundary 
with an axial sequence $(z_n)_n$. 
We denote by $\vp_n\colon\tilde{\Ga}\to\Ga$ the homomorphism 
that is identity on $\Ga$ and sends $z$ to $z_n\in\Ga$. 
A simple ping-pong argument on $X$ shows that $(\vp_n)_n$ 
is asymptotically injective and hence $\Ga$ is mixed-identity-free. 
More precisely, the following is true. 

\begin{lem}\label{lem:pingpong}
Let $g=a_1z^{\epsilon_1} \cdots z^{\epsilon_l}a_{l+1}$ 
be an element in the reduced form with $l\geq1$ 
and $x\in X \setminus\{a_{l+1}^{-1} z_{\bar{\epsilon}_l}\}$, 
where $\bar{\epsilon}$ is the opposite of $\epsilon$. 
Then $\vp_n(g) x \to a_1z_{\epsilon_1}$ as $n\to\infty$.
\end{lem}

We denote by $\sigma$ the corresponding action of $\Ga$ 
on $C(X) \subset C(X)^\cU$, where the embedding 
is the diagonal embedding.
We denote by $\sigma_z$ the automorphism 
arising from the sequence $(\sigma_{z_n})_n$.
They together give rise to the action 
of $\tilde{\Ga}=\Ga*\ip{z}$ on $C(X)^\cU$, 
which is given by
$\sigma_g([f_n]_n) = [\sigma_{\vp_n(g)}(f_n)]_n$.
We are interested in the $\tilde{\Ga}$-invariant 
$\mathrm{C}^*$-subalgebra 
\[
C(X_\Tree) \coloneq \mathrm{C}^*( \bigcup_{g\in\tilde{\Ga}} \sigma_g( C(X) ) ) \subset C(X)^\cU,
\] 
where $X_\Tree$ is the Gelfand spectrum of $C(X_\Tree)$.
By definition, there is a surjective homomorphism 
\[
Q\colon C(X^\Tree) \cong \bigotimes_{t\in\Tree} C(X)
 \to C(X_\Tree),\quad Q(f^{(s)}) = \sigma_{\rho(s)}(f),
\]
where $f^{(s)}\in C(X^\Tree)$ 
is defined for $f\in C(X)$ 
by $f^{(s)}(\mathbf{x}) \coloneq f(x_s)$, 
$\mathbf{x}=(x_t)_t$. 
This gives rise to an embedding 
$Q_*\colon X_T\to X^\Tree$.
We equip $X^\Tree$ a $\tilde{\Ga}$-action by 
declaring that $a\in\Ga$ acts at $\mathbf{x}=(x_t)_t\in X^\Tree$ 
by $( a \mathbf{x} )_o \coloneq ax_o$ 
and $(a\mathbf{x})_t \coloneq x_{a^{-1}t}$ for $t\in\Tree\setminus\{o\}$, 
and that $z$ acts by $(z\mathbf{x})_t \coloneq x_{z^{-1}t}$. 

\begin{lem}
The embedding $Q_*\colon X_\Tree\to X^\Tree$ 
is $\tilde{\Ga}$-equivariant.
\end{lem}
\begin{proof}
Let $\varsigma$ denote the corresponding 
$\tilde{\Ga}$-action on $C(X^\Tree)$. 
It is enough to check that 

\noindent\begin{minipage}[b]{.95\textwidth}
\[
Q(\varsigma_a (f^{(o)})) = Q( \sigma_a(f)^{(o)})
 = \sigma_a(f) = \sigma_a(Q(f^{(o)})),
\] 
\[
Q(\varsigma_a (f^{(s)})) = Q( f^{(as)}) = \sigma_{\rho(as)}(f)
 = \sigma_a(\sigma_{\rho(s)}(f)) = \sigma_a(Q (f^{(s)})) 
\]
for $s\neq o$ and, 
\[
Q(\varsigma_z (f^{(s)})) = Q( f^{(zs)}) = \sigma_{\rho(zs)}(f)
 = \sigma_z(\sigma_{\rho(s)}(f)) = \sigma_z(Q(f^{(s)})).
\]
\end{minipage}
\end{proof}

Hereafter, we omit writing $Q$ and identify $\varsigma$ 
with $\sigma$. 
Thus we view $X_\Tree$ as a compact subset of $X^\Tree$.
We view the label map $\ell$ as a map taking values 
in $X$ by identifying $\Ga \{ z^\pm \}$ with $\Ga \{ z_\pm\}$. 
For each $s\in\Tree$, 
we define the continuous embedding 
$\theta(s,\,\cdot\,)\colon X\to X^\Tree$ by 
$\theta(s,x)_s \coloneq x$ and $\theta(s,x)_t \coloneq \ell(t,s)$ for $t\neq s$.
We also define the continuous embedding  
$\theta\colon\partial\Tree\to X^\Tree$ by 
$\theta(\omega)_t \coloneq \ell(t,\omega)$.
Note that $\omega\to\theta(\omega)_t$ is locally constant for every $t$. 

The space $X_\Tree$ has a kind of 
``tree-graded'' structure (\cite{ds}), as follows. 
See Figure~\ref{fig:2}. 
\begin{figure}
\includegraphics[height=4cm]{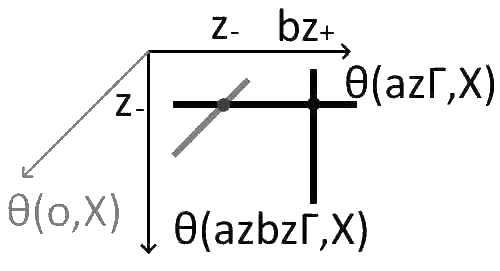}
\caption{The space $X_\Tree$ around 
$\theta(az\Ga,bz_+)=\theta(azbz\Ga,z_-)$.}\label{fig:2}
\end{figure}

\begin{thm}\label{thm:space}
In the above setting for $\tilde{\Ga}=\Ga*\ip{z}\acts\Tree$, the following hold.
\begin{enumerate}
\item\label{p1}
$X_\Tree = \theta(\Tree,X) \cup \theta(\partial\Tree)$.
\item\label{p2}
$\theta(\Tree,X) \cap \theta(\partial\Tree) = \emptyset$.
\item\label{p3}
For $s\neq t$ and $x,y\in X$, 
one has $\theta(s,x)=\theta(t,y)$ if and only if 
$s$ and $t$ are adjacent and $x=\ell(s,t)$ and $y=\ell(t,s)$. 
\item\label{p4}
For $a\in\Ga$, one has 
$a\theta(o,x)=\theta(o,ax)$ and $a\theta(s,x)=\theta(as,x)$ 
for $s\neq o$.
\item\label{p5}
$z\theta(s,x)=\theta(zs,x)$. 
\item\label{p6}
$\theta\colon \partial\Ga\to X_\Tree$ is $\tilde{\Ga}$-equivariant.
\item\label{p7}
A sequence $(\theta(t_n,x_n))_n$ 
converges to $\theta(s,x)$ in $X_\Tree$ if and only if 
either 
$(t_n)_n$ converges to $s$ in $\bar{\Tree}$ and 
the sequence $(y_n)_n$, defined by 
$y_n= x_n$ if $t_n=s$ and $y_n = \ell(s,t_n)$ if $t_n\neq s$,
converges to $x$ in $X$; 
or there is an adjacent point $s_1$ to $s_0 \coloneq s$ in $\Tree$ 
such that $t_n\in\{s_0,s_1\}$ eventually, $x=\ell(s_0,s_1)$, 
and the subsequences $(x_n)_{\{n : t_n=s_i\}}$ 
(possibly there is only one of them) converge 
to $\ell(s_i,s_j)$ for $\{i,j\}=\{0,1\}$.  
A sequence $(\theta(t_n,x_n))_n$ 
converges to $\theta(\omega)$ in $X_\Tree$ if and only if 
$(t_n)_n$ converges to $\omega$ in $\bar{\Tree}$.
\item\label{p8}
A sequence $\theta(\omega_n)$ 
converges to $\theta(s,x)$ in $X_\Tree$ if and only if 
$(\omega_n)_n$ converges to $s$ in $\bar{\Tree}$ and 
and $(\ell(s,\omega_n))_n$ converges to $x$ in $X$.
\item\label{p9}
One has 
$\theta(\Tree,\Ga\{z_\pm\})=\tilde{\Ga}\theta(o,\{z_\pm\})$. 
The map 
$\beta\colon X_\Tree \setminus \tilde{\Ga}\theta(o,\{z_\pm\}) 
\to\bar{\Tree}$, given by $\beta(\theta(s,x))=s$ and $\beta(\theta(\omega))=\omega$, is 
$\tilde{\Ga}$-equivariant and continuous. 
\end{enumerate}
\end{thm}

\begin{proof}
We prove (\ref{p3}). The proof of $(\ref{p2})$ is similar. 
Suppose that $s$ and $t$ are adjacent 
and $x=\ell(s,t)$ and $y=\ell(t,s)$. 
Then $\theta(s,x)_s=x=\ell(s,t)=\theta(t,y)_s$, 
$\theta(s,x)_t=\ell(t,s)=y=\theta(t,y)_t$, and 
$\theta(s,x)_r = \theta(r,s)=\theta(r,t)=\theta(t,y)_r$ 
for $r\in\Tree\setminus\{s,t\}$, because $s$ and $t$ 
are on the same side of $r$. 
This shows $\theta(s,x)=\theta(t,y)$. 
Conversely, let $(s,x)$ and $(t,y)$ be given. 
Suppose that there is $r\in\Tree$ strictly between $s$ and $t$. 
Then $\theta(s,x)\neq\theta(t,y)$ since 
$\theta(s,x)_r = \ell(r,s) \neq \ell(r,t) = \theta(t,y)_r$. 
Suppose next that $s$ and $t$ are adjacent. 
Then $\theta(s,x)_s=x$, $\theta(s,x)_t=\ell(t,s)$ 
and $\theta(t,y)_s=\ell(s,t)$, $\theta(t,y)_t=y$. 
Hence $\theta(s,x)=\theta(t,y)$ only if $x=\ell(s,t)$ and $y=\ell(t,s)$.
This proves (\ref{p3}). 

The assertions (\ref{p4}-\ref{p6}) 
follow from Lemma~\ref{lem:labelequiv}; 
E.g., for $a\in\Ga$ and $s\in\Tree\setminus\{ o\}$, one has
\[
(a\theta(s,x))_o=a\theta(s,x)_o=a\ell(o,s)=\ell(o,as)=\theta(as,x)_o
\]
and for $t\neq o$ 
\[
(a\theta(s,x))_t = \theta(s,x)_{a^{-1}t}
 = \left\{\begin{array}{ll}
     x & \mbox{ if $t=as$} \\
     \ell(a^{-1}t,s) = \ell(t,as) & \mbox{ if $t\neq as$} 
   \end{array}\right\}
 =\theta(as,x)_t.
\]

The ``if'' part of (\ref{p7}-\ref{p8}) are easy. 
For the ``only if'' part, suppose that 
$(\theta(t_n,x_n))_n$ converges to $\theta(s,x)$. 
If $(t_n)_n$ is converges to some point $s'$, 
then by the ``if'' part and the compactness of $X$, 
one has $\theta(t_n,x_n) \to \theta(s',y)$, 
where $y$ is the limit of $y_n \coloneq x_n$ if $t_n=s'$ and 
$y_n \coloneq \ell(s',t_n)$ if $t_n\neq s'$. 
Since $\theta(s',y)=\theta(s,x)$, one has by (\ref{p3}) 
either $s'=s$ or $s'$ is an adjacent point to $s$ 
that satisfies $x=\ell(s,s')$ and $y=\ell(s',s)$. 
This proves the ``only if'' part for the case $(t_n)_n$ is convergent. 
Now consider the case where $(t_n)_n$ is not convergent. 
By the previous case, the only limit points are 
either $s$ or the unique adjacent point $s'$ 
that satisfies $x=\ell(s,s')$. 
The other claimed conditions also follow. 
The rest of the proof of (\ref{p7}-\ref{p8}) is similar. 

We prove (\ref{p1}). 
We first prove that $\theta(\Tree,X)\subset X_\Tree$. 
Let $s\in\Tree$ and $x\in X\setminus\Ga\{z_\pm\}$ be given.
We consider the character $\chi$ on 
$C(X_\Tree)\subset C(X)^\cU$ 
arising from the sequence $(\vp_n(\rho(s))x)_n$. 
Then by Lemma~\ref{lem:pingpong} one has 
\[
\chi(f^{(t)}) = \lim\nolimits_{\cU} \sigma_{\vp_n(\rho(t))}(f)(\vp_n(\rho(s))x) 
 = \lim\nolimits_{\cU} f(\vp_n(\rho(t)^{-1}\rho(s))x) 
 = f(\ell(t,s))
\]
for $t\neq s$ and $\chi(f^{(s)})=f(x)$. 
This means $\chi$ corresponds to 
the point $\theta(s,x)$ in $X_\Tree$. 
Since $\theta(s,\,\cdot\,)\colon X\to X^\Tree$ is 
continuous and $X\setminus\Ga\{z_\pm\}$ is dense in $X$, 
one sees that $\theta(s,X)\subset X_\Tree$. 
Similarly for $\theta(\partial\Tree)$. 
It follows that $\theta(\Tree,X)\cup\theta(\partial\Tree)\subset X_\Tree$. 

It remains to prove the opposite inclusion. 
Let $\mathbf{x} \coloneq (x_t)_t\in X^\Tree$ be given and assume that 
$\mathbf{x} \in X_\Tree$. 
We define $w\colon \Tree \to \Tree$, 
``the orientation of $\mathbf{x}$ at $t$'', as follows. 
If $x_t \in X\setminus\Ga\{z_\pm\}$, then $w(t) \coloneq t$. 
If $x_t \in \Ga\{ z_\pm\}$, then we define $w(t)$ to be the unique 
point that is adjacent to $t$ and satisfies $\ell(t,w(t))=x_t$. 
We claim that if $s,t\in\Tree$ are adjacent, then 
it cannot happen that $w(s)\neq t$ and $w(t)\neq s$ simultaneously. 
Suppose for a contradiction that it happened. 
We may assume that $s$ is between $o$ and $t$ 
and $\rho(t)=\rho(s) a z$ with no cancellation. 
Thus one has $x_s \neq az_+$ and $x_t \neq z_-$. 
Let $f\in C(X)$ be such that 
$f(x_s)=1$ and $f=0$ on a neighborhood of $az_+$. 
Similarly let $g$ be such that $g(x_t)=1$ and $g=0$ 
on a neighborhood of $z_-$. 
We consider the function $f^{(s)}g^{(t)}$ on $X^\Tree$. 
It satisfies $(f^{(s)}g^{(t)})( \mathbf{x} ) = f(x_s)g(x_t) = 1$.
However, 
\[
(f^{(s)}g^{(t)})|_{X_\Tree} = Q(f^{(s)}g^{(t)}) 
 = \sigma_{\rho(s)}(f)\sigma_{\rho(t)}(g)
 = \sigma_{\rho(s)}( f \sigma_{az}(g) ) = 0
\]
since $(\supp f)  \cap az_n(\supp g) = \emptyset$ eventually. 
This means that $\mathbf{x} \in X^\Tree\setminus X_\Tree$, 
a contradiction. 
Thus we have proved the claim for adjacent 
points $s,t\in\Tree$ that $w(s)\neq t$ implies $w(t)=s$. 
This yields three possibilities. 
(1): There is $s\in \Tree$ such that $w$ is oriented to 
the $s$ direction. 
(2): There is $\omega\in\partial\Tree$ such that 
$w$ is oriented to the $\omega$ direction. 
(3): There is a pair $s_1,s_2$ of adjacent points such that 
$w$ is oriented to $\{s_1,s_2\}$ direction 
and $w(s_i)=s_j$ for $\{ i, j \}=\{1,2\}$. 
In the first case, one has $\mathbf{x}=\theta(s,x_s)$. 
In the second case, one has $\mathbf{x}=\theta(\omega)$. 
In the third case, one has $\mathbf{x}=\theta(s_1,\ell(s_1,s_2))=\theta(s_2,\ell(s_2,s_1))$.
It follows that $\theta(\Tree,X)\cup\theta(\partial\Tree)\supset X_\Tree$. 

Finally, the assertion (\ref{p9}) is a consequence of (\ref{p1}-\ref{p8}).
\end{proof}

\section{Proof of Theorem~\ref{thm:main}}
\begin{proof}[Proof of Theorem~\ref{thm:main}]
By Theorem 2.6 in \cite{robert}, $\mathrm{C}^*$-selflessness 
follows from the latter statement. 
Let $\Ga\acts X$ be an extreme boundary 
with an axial sequence $(z_n)_n$. 
(The reason for sticking with sequences 
rather than nets is purely aesthetic. 
If a countable group $\Ga$ has a topologically-free 
extreme boundary, then it has one which is second countable.)
%
%
We denote by $\pi$ the corresponding homomorphism 
from $\tilde{\Ga} \coloneq \Ga*\ip{z}$ into 
$\mathrm{C}^*_\lambda(\Ga)^\cU \subset \IB(\ell_2\Ga)^\cU$, 
which is viewed as a unitary representation on 
some Hilbert space $\cH$ 
via a faithful representation of $\IB(\ell_2\Ga)^\cU$. 
We take a $\Ga$-equivariant embedding 
$C(X)\subset\ell_\infty\Ga\subset\IB(\ell_2\Ga)$. 
It gives rise to a $\tilde{\Ga}$-equivariant embedding 
of $C(X_\Tree)$ into $\IB(\ell_2\Ga)^\cU$, 
where $X_\Tree$ is defined in the previous section. 
By (\ref{p9}) in Theorem~\ref{thm:space}, 
there is a unital $\tilde{\Ga}$-equivariant 
homomorphism from $C(\bar{\Tree})$ into 
the $\mathrm{C}^*$-alge\-bra $B(X_\Tree)$ 
of bounded Borel functions on $X_\Tree$, 
where $\tilde{\Ga}\theta(o,\{z_\pm\})$ is matched 
with some free $\tilde{\Ga}$-orbits in $\bar{\Tree}$. 
As the representation 
of $C(X_\Tree)$ on $\cH$ uniquely extends 
to a $\sigma$-normal representation of $B(X_\Tree)$, 
one finds a $\tilde{\Ga}$-equivariant representation of 
$C(\bar{\Tree})$ on $\cH$.  
Thus by Theorem~\ref{thm:repn} 
and the fact that $\pi|_\Ga$ is continuous 
on $\mathrm{C}^*_\lambda(\Ga)$, the representation 
$\pi$ is continuous on $\mathrm{C}^*_\lambda(\tilde{\Ga})$. 
\end{proof}

%
%
%
\section{Direct products of $\mathrm{C}^*$-selfless groups}\label{sec:dp}
It is not hard to see that if $A\subset B$ is an \emph{existential} 
embedding of $\mathrm{C}^*$-alge\-bras (see Section~1 in \cite{robert}), 
then so is $A\otimes C \subset B\otimes C$, provided that 
the $\mathrm{C}^*$-alge\-bra $C$ is \emph{exact}. 
(Without the exactness assumption, 
the canonical embedding $A^\cU\otimes C \subset (A\otimes C)^\cU$
is not continuous w.r.t.\ the minimal tensor product. See \cite{bo}.) 
Hence, if $A_i\subset B_i$ are existential embedding and $B_i$ are exact, 
then the corresponding 
embedding $A_1\otimes A_2 \subset B_1\otimes B_2$ is existential. 
The same holds true in the $\mathrm{C}^*$-prob\-a\-bility space setting. 

We denote by $(\cT,\omega)$ the Toeplitz $\mathrm{C}^*$-prob\-a\-bility space 
consisting of the Toeplitz algebra generated 
by the unilateral shift $T$ on $\ell_2(\{0,1\,\ldots\})$ 
and the vacuum state $\omega$ associated with the vacuum vector $\delta_0$. 
We denote by $(\cC_1,\tau)\subset(\cT,\omega)$ 
the $\mathrm{C}^*$-prob\-a\-bility space generated by $s \coloneq (T+T^*)/2$ in $\cT$. 
Note that $s$ corresponds to the identity function under the 
isomorphism $(\cC_1,\tau) \cong (C([-1,1]),\gamma)$, 
where the state $\gamma$ is given by 
by the semicircle probability measure $\frac{2}{\pi}\sqrt{1-s^2}\,\mathrm{d}s$. 
See Section 2.6 in \cite{vdn}.

We denote by $\bF_\infty$ the free group on 
generators $\{ t_i \}_{i=1}^\infty$. We view it as embedded 
in the reduced group $\mathrm{C}^*$-alge\-bra 
$\mathrm{C}^*_\lambda(\bF_\infty)$ acting on $\ell_2\bF_\infty$. 
It forms the $\mathrm{C}^*$-prob\-a\-bility space 
$(\mathrm{C}^*_\lambda(\bF_\infty),\tau)$, 
where $\tau$ is the canonical tracial state 
that is given by $\tau(g)=\delta_{g,1}$ for $g\in\bF_\infty$.
It is well-known (\cite{vdn}) that the sequence 
\[
s_n = (8n)^{-1/2}\sum_{i=1}^n (t_i+t_i^{-1})
\]
in $(\mathrm{C}^*_\lambda(\bF_\infty),\tau)$ 
\emph{strongly converges} to $s$ in $(\cC_1,\tau)$. 
Namely, for every polynomial $p$, one has 
$\tau(p(s_n))\to\gamma(p(s))$ and $\|p(s_n)\|\to \|p(s)\|$. 
We will prove a vectorial version of this fact in Theorem~\ref{thm:crossedproduct}. 
Let $(A,\vp)$ be a $\mathrm{C}^*$-prob\-a\-bility space on which $\bF_\infty$ acts 
(leaving $\vp$ invariant). 
Then the reduced crossed product $A\rtimes_\mathrm{r}\bF_\infty$ 
is equipped with the canonical state, still denoted by $\vp$, that is given by 
$\vp( a g) = \vp(a)\tau(g)$ for $a\in A$ and $g\in\bF_\infty$. 

Let $(A,\vp)$ be a $\mathrm{C}^*$-prob\-a\-bility space. 
We view $A$ as embedded in the ultrapower 
$(A,\vp)^\cU \coloneq (A^\cU,\vp^\cU)$ diagonally. 
(The difference in the free ultrafilter $\cU$ will have no effect on the discussion.) 
Consider a bounded sequence $(a_n)_{n=1}^\infty$ of self-adjoint elements 
in $(A,\vp)$ that is strongly convergent to $a$ in $(\mathrm{C}^*(\{1,a\}), \psi)$. 
We say the sequence $(a_n)_{n=1}^\infty$ is \emph{strongly asymptotically free} 
(w.r.t.\ $\cU$) from a $\mathrm{C}^*$-subalgebra $A_0$ if the induced homomorphism 
\[
(A_0,\vp|_{A_0})\ast(\mathrm{C}^*(\{1,a\}),\psi)\to (A^\cU,\vp^\cU)
\]
is well-defined and continuous on the reduced free product. 
Note that the reduced free product is compatible with embedding  (see \cite{bd} or Corollary 4.8.4 in \cite{bo}.

\begin{thm}\label{thm:crossedproduct}
Let $(A,\vp)$ be a $\mathrm{C}^*$-prob\-a\-bility space on 
which $\bF_\infty$ acts by $\sigma$. 
Assume that 
\[
\forall a\in A\quad \frac{1}{2n}\sum_{i=1}^n (\sigma_{t_i}(a)+\sigma_{t_i}^{-1}(a)) \to \vp(a).
\]
Then, the sequence $(s_n)_{n=1}^\infty$ as above in $A\rtimes_\mathrm{r}\bF_\infty$ 
is strongly asymptotically free from $A$ (in fact from $A\rtimes_\mathrm{r}\bF_\infty$).
\end{thm}
\begin{proof}
The proof is inspired from \cite{vdn} and \cite{hp}. 
We can write $\bF_\infty$ as a disjoint union 
\[
\bF_\infty=\{1\}\cup\bigcup_i (E_i^+\cup E_i^{-}),
\]
where $E_i^\pm$ is the set of those elements  
whose reduced form starts with (a positive power of) $t_i^\pm$. 
We denote by $P_{i,\pm}$ the orthogonal projection on $\ell_2\bF_\infty$ 
that corresponds to $E_i^\pm$. They are mutually orthogonal. 
One has
$t_i = P_{i,+}t_i + t_iP_{i,-} \eqcolon u_{i,+} + u_{i,-}^*$ (see Proof of Proposition 1.1 in \cite{hp}). 
Note that $u_{i,\pm}$ are partial isometries such that 
$u_{i,\pm}^*u_{i,\pm}=1-P_{i,\mp}$ and $u_{i,\pm}u_{i,\pm}^*=P_{i,\pm}$. 
Hence, for 
\[
T_n \coloneq (2n)^{-1/2}\sum_{i=1}^n (u_{i,+}+u_{i,-}),
\] 
one has $s_n = (T_n+T_n^*)/2$ and 
$T_n^*T_n = 1 - (2n)^{-1}\sum_{i=1}^n (P_{i,-}+P_{i,+}) \approx 1$.

We represent 
$A\rtimes_\mathrm{r}\bF_\infty$ on 
the Hilbert space 
$\cH \coloneq L^2(A,\vp)\otimes\ell_2\bF_\infty$ by 
$g\mapsto (1\otimes g)$ for $g\in \bF_\infty$ and 
$a\mapsto \pi(a)$, 
$\pi(a)(\xi\otimes\delta_h) = \sigma_h^{-1}(a)\xi \otimes\delta_h$, 
for $a\in A$. 
Then, by the assumption, 
\[
(1\otimes T_n)^*\pi(a)(1\otimes T_n) 
= \pi(\frac{1}{2n}\sum_{i=1}^n (\sigma_{g_i}^{-1}(a)(1-P_{i,-})+\sigma_{g_i}(a)(1-P_{i,+}))
\to \vp(a)
\]
for $a\in A$. 
Hence $T \coloneq [(1\otimes T_n)]_n \in \IB(\cH)^\cU$ 
is an isometry that satisfies 
$T^*\pi(a) T=\vp(a)$ for $a\in A$. 
We denote by $\omega$ the vector state 
on $\IB(\cH)$ associated 
with the vector $\hat{1}\otimes\delta_1$. 
Thus 
$(A,\vp)\subset(\IB(\cH),\omega)$. 
The isometry $T$ in $(\IB(\cH),\omega)^\cU$ 
generates a copy of the Toeplitz probability space $(\cT,\omega)$. 
In particular, $[s_n]_n = (T+T^*)/2$ generates a copy of 
$(\cC_1,\tau)$ in $(A\rtimes\bF_\infty,\vp)^\cU$.  
It is rather easy to see that $A$ and $\mathrm{C}^*(T)$ are 
free from each other w.r.t.\ $\omega^\cU$. 
It remains to show continuity of the induced homomorphism 
from the reduced free product $(A,\vp)\ast(\cT,\omega)$ 
into $(\IB(\cH),\omega)^\cU$. 

The universal $\mathrm{C}^*$-alge\-bra generated by $A$ and an isometry $T$ 
that satisfy $T^*aT=\vp(a)$ is the Toeplitz--Pimsner algebra $\cT_{A,\vp}$ 
over the Hilbert $A$-module $\cspan ATA$ 
(see, e.g., Example 4.6.11 in \cite{bo}). 
Since the canonical conditional expectation from $\cT_{A,\vp}$ 
onto $A$ is non-degenerate (Theorem 4.6.6 in \cite{bo}), 
the Toeplitz--Pimsner algebra 
$\cT_{A,\vp}$ is isomorphic to the reduced free product 
$(A,\vp)*(\cT,\omega)$. 
This proves the desired continuity.
\end{proof}

\begin{proof}[Proof of Theorem~\ref{thm:tensor}]
By Theorem 4.1 in \cite{robert}, we may assume that $i\in\{1,2\}$. 
Since $(A_i,\vp_i)$ are selfless, the embeddings 
\[
(A_i\otimes\vp_i) 
 \subset (A_i,\vp_i)\ast(\mathrm{C}^*_\lambda(\bF_\infty),\tau)
 \eqcolon (B_i,\psi) \cong (\bigast_{\bF_\infty} (A_i,\vp_i))\rtimes_{\mathrm{r}}\bF_\infty
\] 
are existential (Theorem 2.6 in \cite{robert}). 
The $\mathrm{C}^*$-alge\-bras $B_i$ are exact by Dykema's theorem (\cite{dykema} or Corollary 4.8.3 in \cite{bo}).
Hence the embedding of 
$(A_1\otimes A_2,\vp_1\otimes\vp_2) $ into $(B_1\otimes B_2,\psi_1\otimes\psi_2)$ 
is existential. 
We write $\bF_\infty^{(i)}$ for the copy of $\bF_\infty$ in $B_i$. 
Note that 
\begin{align*}
(B_1\otimes B_2,\psi_1\otimes\psi_2) 
 &\cong (\bigast_{\bF_\infty^{(1)}} (A_1,\vp_1) \otimes \bigast_{\bF_\infty^{(2)}} (A_2,\vp_2))\rtimes_{\mathrm{r}}(\bF_\infty^{(1)}\times\bF_\infty^{(2)})\\
 &\supset (A \rtimes_{\mathrm{r}}\bF_\infty,\vp),
\end{align*}
where $(A,\vp) \coloneq \bigast_{\bF_\infty^{(1)}} (A_1,\vp_1) \otimes \bigast_{\bF_\infty^{(2)}}(A_2,\vp_2)$ 
and 
$\bF_\infty\subset \bF_\infty^{(1)}\times\bF_\infty^{(2)}$ is the diagonal. 
We will verify that the assumption of Theorem~\ref{thm:crossedproduct} 
is satisfied. 
Once this is done, the proof of Theorem~\ref{thm:tensor} is complete by 
Lemma 1.2, Lemma 2.4, and Theorem 2.6 in \cite{robert}. 
Since $\ker\vp$ is densely spanned by those elements 
$a_1\otimes 1$, $1\otimes a_2$, and $a_1\otimes a_2$, 
where $a_i\in\ker\vp\cap \bigast_{F} (A_i,\vp_i)$ 
for some finite subset $F\subset \bF_\infty^{(i)}$,
the assumption of Theorem~\ref{thm:crossedproduct} follows from 
Voiculescu's inequality (\cite{voiculescu}, and \cite{junge} for the vectorial version) 
that 
\[
\| \sum_{j=1}^m x_j \otimes y_j \| \le 3(\sum_{j=1}^m \|x_j\|^2\|y_j\|^2)^{1/2} 
\precsim m^{1/2}\max \|x_j\|\|y_j\|
\]
whenever $x_j$'s (or $y_j$'s) are freely independent and mean zero. 
\end{proof}
\section{Completely selfless $\mathrm{C}^*$-alge\-bras}\label{sec:cs}
In this section, we indicate the way to circumvent 
the exactness assumption in Theorem~\ref{thm:tensor}. 
Recall that an embedding $A\subset B$ of $\mathrm{C}^*$-alge\-bras 
is said to be \emph{existential} if there is an ultrafilter $\cU$ and 
an embedding $\sigma\colon B\hookrightarrow A^\cU$ 
whose restriction to $A$ is the diagonal embedding of $A$ into $A^\cU$. 
We say that the embedding is \emph{completely existential} if 
there is $\sigma$ as above that moreover satisfies that 
for every $\mathrm{C}^*$-alge\-bra $C$ 
the induced homomorphism $B\otimes C \to (A\otimes C)^\cU$ 
is continuous. By the Effros--Haagerup lifting theorem, 
the additional condition is equivalent to that 
the homomorphism $\sigma$ is locally liftable to 
completely positive contractions into the $\ell_\infty$-sum $\prod A$. 
Namely, $A\subset B$ is completely existential if and only if 
for every finite dimensional operator system $E\subset B$ 
and $\ve>0$, there is a unital completely positive 
map $\theta\colon E\to A$ satisfying that 
$\|\theta(x) \| \geq (1-\ve)\|x\|$ for $x\in E$, 
$\|\theta(xy) - \theta(x)\theta(y)\|<\ve$ 
for those $x,y\in E$ such that $xy\in E$, 
and $\|\theta(a) -a \|<\ve$ for $a\in E\cap A$. 
In the $\mathrm{C}^*$-prob\-a\-bility space 
setting $(A,\vp)\subset(B,\psi)$, it is moreover required 
that $\|\vp\circ\theta - \psi|_E\|<\ve$.
It is not hard to see that the class of completely existential 
embeddings are closed under compositions and tensor products. 
It is also closed under reduced free products (cf.\ Corollary 1.9 in \cite{robert}).

\begin{lem}
For every completely existential embedding $(A,\vp)\subset (B,\psi)$ of 
$\mathrm{C}^*$-prob\-a\-bility spaces and 
every $(C,\rho)$,  
the embedding $(A,\vp)\ast(C,\rho)\subset (B,\psi)\ast(C,\rho)$ 
is completely existential. 
\end{lem}
\begin{proof}
Wthout loss of generality, we assume that $B$ is separable. 
Take an increasing sequence of finite dimensional operator 
systems $E_n\subset B$ such that $\bigcup E_n$ is dense in $B$ 
and $\bigcup (E_n\cap A)$ is dense in $A$. 
Since the embedding $(A,\vp)\subset(B,\psi)$ 
is completely existential, for each $n$, there is a unital completely positive map $\theta_n\colon E_n\to A$ 
that is approximately isometric, approximately multiplicative, 
approximately state preserving, and $\theta_n|_{E_n\cap A}\approx\id_{E_n\cap A}$.
It extends to a unital completely positive 
map, still denoted by $\theta_n$, 
from $B$ into $\IB(L^2(A,\vp))$. 
We denote by $\omega$ the vector state that corresponds to the 
GNS-vector $\hat{1}$ in $L^2(A,\vp)$. 
One has $\psi_n \coloneq \omega\circ\theta_n\to\psi$ pointwise on $B$. 
Let $B_n$ denote the quotient of $B$ by the GNS-kernel of $\psi_n$. 
By the Choda--Blanchard--Dykema theorem (\cite{bd} 
or Theorem 4.8.5 in \cite{bo}), $\theta_n\ast\id_C$ 
defines a unital completely positive map from 
$(B_n,\psi_n)\ast(C,\rho)$ into $(\IB(L^2(A,\vp)),\omega)\ast(C,\rho)$. 
By the Skoufranis--Pisier theorem 
(Theorem 7.1 in \cite{pisier}) for amalgamated free products, 
the (a priori formally defined) embedding 
\[
(B,\psi)\ast (C,\rho) \hookrightarrow \prod_{n\in \IN}(B_n,\psi_n)\ast(C,\rho)/\cU
\]
is continuous and locally liftable. 
By composing this with $[\theta_n\ast\id_C]_n$, 
one obtains a locally liftable embedding 
of $(B,\psi)\ast (C,\rho)$ into $((\IB(L^2(A,\vp)),\omega)\ast(C,\rho))^\cU$, 
which lands in $((A,\vp)\ast(C,\rho))^\cU$. 
This witnesses existentialness. 
\end{proof}

We say a $\mathrm{C}^*$-prob\-a\-bility space 
\emph{completely selfless} if the ``complete'' analogue of 
the equivalent conditions 
in Theorem 2.6 in \cite{robert} hold. 
A group is \emph{completely $\mathrm{C}^*$-selfless} if its 
reduced group $\mathrm{C}^*$-alge\-bra is completely selfless.
Note that no nuclear tracial $\mathrm{C}^*$-prob\-a\-bility 
space is completely selfless, 
since if it were, then the tracial 
state on $\mathrm{C}^*_\lambda(\bF_\infty)$ 
would be \emph{amenable} (see Chapter 6 in \cite{bo}), 
which is absurd. 
Theorems~\ref{thm:main} and \ref{thm:tensor} 
are upgraded as follows. 

\begin{thm}
An infinite countable discrete group having 
a topologically-free extreme boundary is completely 
$\mathrm{C}^*$-selfless. 
The tensor product of separable and completely selfless 
$\mathrm{C}^*$-prob\-a\-bility spaces is completely selfless. 
In particular, the class of countable discrete groups that 
are completely $\mathrm{C}^*$-selfless is closed under 
direct product.
\end{thm}

The proof is same as Theorems~\ref{thm:main} 
and \ref{thm:tensor}, but uses the following fact 
(\cite{shlyakhtenko} or 
Example 4.6.11 and Exercise 4.8.1 in \cite{bo}).

\begin{lem}\label{lem:toeplitz}
Let $(A,\vp)$ be a $\mathrm{C}^*$-prob\-a\-bility space, 
$(\cT,\omega)$ be the Toeplitz $\mathrm{C}^*$-prob\-a\-bility space, 
and $C$ be a unital $\mathrm{C}^*$-alge\-bra. 
Then $C\otimes ((A,\vp)\ast(\cT,\omega))$ is 
the universal $\mathrm{C}^*$-alge\-bra generated by 
$C\otimes A$ and an isometry $T$ 
(the generator of the Toeplitz algebra) that satisfies 
$T^*(c\otimes a)T=\vp(a)(c\otimes1)$ for $a\in A$ and $c\in C$. 
Moreover, the free product state $\psi$ on $(A,\vp)\ast(\cT,\omega)$ 
is the unique state that satisfies $\psi|_A=\vp$ and $\psi(aTT^*a^*)=0$ for $a\in A$.
\end{lem}
%
%

We do not use it, but note the fact that 
$(A,\vp)\ast(\cT,\omega)$ is 
simple and purely infinite, provided that
$A \cap \IK(L^2(A,\vp))=0$ (\cite{kumjian}). 

\begin{thm}\label{thm:toeplitz}
Let $(A,\vp)$ be a $\mathrm{C}^*$-prob\-a\-bility space. 
Suppose that there are a $\mathrm{C}^*$-prob\-a\-bility 
space $(B,\psi)$ containing $(A,\vp)$, an ultrafilter $\cU$, 
and an isometry $T\in B^\cU$ that satisfies that 
$T^*aT=\vp(a)$ and $\psi^\cU(aTT^*a^*)=0$ for all $a\in A$ and 
that $(T+T^*)/2 \in A^\cU$. 
Then $(A,\vp)$ is completely selfless. 
\end{thm}
\begin{proof}
Let $(A_1,\vp_1) \subset (A_2,\vp_2)$ 
denote $(A,\vp)*(\cC_1,\tau)\subset (A,\vp)*(\cT,\omega)$, 
see Section~\ref{sec:dp}.
Let a $\mathrm{C}^*$-prob\-a\-bility space $(C,\theta)$ be given. 
We have to show there is an embedding 
\[
(C\otimes A_1,\theta\otimes\vp_1)
 \hookrightarrow((C\otimes A)^\cU,(\theta\otimes\vp)^\cU)
\] 
that extends the diagonal embedding of $C\otimes A$.
We view $T$ as $1_C\otimes T$ in $(C\otimes B)^\cU$.
Then by Lemma~\ref{lem:toeplitz}, the isometry $T$ gives rise to 
a representation of $C\otimes A_2$ in $(C\otimes B)^\cU$. 
Since the state $(\theta\otimes\psi)^\cU$ on $(C\otimes B)^\cU$ 
restricts to $\theta\otimes\vp_2$ on $C\otimes A_2$, 
this representation is a state-preserving embedding of 
$C\otimes A_2$ into $(C\otimes B)^\cU$. 
Since the generator $(T+T^*)/2$ of $\cC_1$ belongs to $A^\cU$, 
it restricts to a desired embedding of $C\otimes A_1$ into $(C\otimes A)^\cU$.
\end{proof}

\section{Proof of Theorem~\ref{thm:general}}

\begin{proof}[Proof of Theorem~\ref{thm:general}]
We deal with the first case. 
By Glimm's lemma, $\vp$ is approximated on 
any finite subset $F$ of $A$ by a pure state $\psi$ 
that are disjoint from $\vp$. 
By excision (see Proposition 11.4.2 in \cite{bo}), 
there is a net of projections $p_j$ such that 
$\| p_j a p_j - \psi(a)p_j \| \to 0$ for $a\in A$. 
The net $(p_j)_j$ converges in $A^{**}$ to the support 
projection of the pure state $\psi$. 
Hence $\vp(ap_ja^*)\to0$ for $a\in A$. 
Thus, by pure infiniteness, 
there are a directed set $I$ and a net $(T_i)_{i\in I}$ of isometries in $A$ such that 
$\|T_i^*a T_i-\vp(a)\| \to0$ and $\vp(aT_iT_i^*a^*)\to0$ 
for $a\in A$.
Let $\cU$ be a cofinal ultrafilter on $I$. 
The isometry $T \coloneq [T_i]_i$ in $A^\cU$ satisfies $T^*aT=\vp(a)$ and $\vp^\cU(aTT^*a^*)=0$ for $a\in A$. By Theorem~\ref{thm:toeplitz}, 
we are done.

We deal with the second case. 
Since $A$ is $\cZ$-stable and $\mathrm{C}^*_\lambda(\bF_\infty)\hookrightarrow \cZ^\cU$ (Theorem 4.1 in \cite{oz}), 
the embedding 
$A\subset A\otimes C^*_\lambda(\bF_\infty)$ 
is existential by exactness, 
where we identify $A$ with $A\otimes\IC1$.
Since $A$ is simple and monotracial, the \emph{uniform Dixmier 
property} holds by the Haagerup--Zsido theorem (\cite{hz}). 
Namely, there is a net $(F_i)_{i\in I}$ of finite families of unitary elements 
in $A$ that satisfies $|F_i|^{-1}\sum_{u\in F_i} uau^*\to\tau(a)$ for $a\in A$. 
By replacing $F_i$ with $F_i^*F_i$, we may assume that 
the finite set $F_i$ is closed under the adjoint operation. 
Let $\{ u^{(i)}_n : n=1,\ldots,|F_i| \}$ be an enumeration of $F_i$. 
Since 
$A\subset A\otimes \mathrm{C}^*_\lambda(\bF_\infty)$ is existential, 
by replacing $u^{(i)}_n$ with likenesses in $A$ of $u^{(i)}_n\otimes t_n$ 
in $A\otimes \mathrm{C}^*_\lambda(\bF_\infty)$, 
we may further assume by Voiculescu's inequality 
(see Proof of Theorem~\ref{thm:tensor}) that 
$\| |F_i|^{-1}\sum_{u\in F_i} uau\| \to0$ for $a\in A$. 
Let $\cU$ be a cofinal ultrafilter on $I$. 

Let's consider the Cuntz algebra $(\cO_\infty,\omega)$ 
generated by isometries $\{l_n\}$ with mutually orthogonal ranges, 
together with the vacuum state. 
Recall that  $\{ (l_n+l_n^*)/2\}$ generates 
the free semicircular system 
$(\cC,\tau) \subset (\cO_\infty,\omega)$ 
(see \cite{vdn}). 
The elements 
\[
T_i \coloneq (2|F_i|)^{-1/2}\sum_n (u^{(i)}_n+(u^{(i)}_n)^*) \otimes l_n
\]
in $A\otimes\cO_\infty$ satisfies $T_i^* a T_i \to \tau(a)$ and 
$(\tau\otimes\omega)(aT_iT_i^*a^*)=0$ for $a\in A$. Moreover, 
$(T_i+T_i^*)/2\in A\otimes\cC$. 
Hence by Lemma~\ref{lem:toeplitz}, it gives rise to an embedding 
of $(A,\tau)\ast(\cT,\omega)$ into $(A\otimes\cO_\infty,\tau\otimes\omega)^\cU$.
It restricts to an embedding of 
$(A,\tau)\ast(\cC_1,\tau)$ into $(A\otimes\cC,\tau\otimes\tau)^\cU$ 
that witnesses the \emph{relative existentialness} (see \cite{robert}) 
of $(A,\tau)\hookrightarrow(A,\tau)\ast(\cC_1,\tau)$ 
in $(A\otimes\cC,\tau\otimes\tau)$. 
Since the embedding of $(A,\tau)$ into $(A\otimes\cC,\tau\otimes\tau)$ 
is existential by assumption, it follows that the embedding of $(A,\tau)$ into 
$(A,\tau)\ast(\cC_1,\tau)$ is also existential. 
\end{proof}
\section{The Powers property and selflessness}\label{sec:pow}
$\mathrm{C}^*$-simplicity of the free group of rank $2$ 
has been proved by R.~T.~Powers (\cite{powers}) by a combinatorial 
method. The combinatorial method is later formalized by 
P.~de la Harpe (\cite{delaharpe}) as 
the \emph{Powers property}. 
The following is very close to it in spirit. 
It is designed to allow the Haagerup--Pisier 
type decomposition (\cite{hp}).

\begin{defn}\label{defn:pow}
We say that a group $\Ga$ has 
property $\mathrm{P}_{\mathrm{PHP}}$ if 
for every finite subset $F\subset\Ga$ and $\ve>0$, 
there is $N\in\IN$ such that for every $n\geq N$, 
there are 
$t_i\in\Ga$ and $C_i\subset D_i\subset\Ga$ for 
$1\le i \le n$ satisfying that 
the members of 
\[
\{aC_i : a\in F,\, i=1,\ldots,n\}\cup\{ bt_j^{-1}(\Gamma\setminus D_j) : b\in F,\, j=1,\ldots,n\}
\] 
are mutually disjoint and that 
\[
\sup_{x\in\Ga} |\{ i : x\in D_i \} \cup \{ j : x \in t_j^{-1} (\Ga\setminus C_j)\}| \le \ve n^{1/2}.
\]
\end{defn}

It is not too hard to see that the class of groups 
with $\mathrm{P}_{\mathrm{PHP}}$ is closed 
under the direct product; for $t_i^{(k)}\in\Ga^{(k)}$ and 
$C_i^{(k)}\subset D_i^{(k)}\subset\Ga^{(k)}$, $k=1,2$, 
set $t_i \coloneq (t_i^{(1)},t_i^{(2)})\in \Ga^{(1)}\times \Ga^{(2)}$, 
$C_i \coloneq C_i^{(1)}\times C_i^{(2)}$, and 
$D_i \coloneq (D_i^{(1)}\times \Ga^{(2)})\cup(\Ga^{(1)}\times D_i^{(2)})$. 

\begin{thm}
A group $\Ga$ with property $\mathrm{P}_{\mathrm{PHP}}$ is completely $\mathrm{C}^*$-selfless. 
\end{thm}
\begin{proof}
Given a finite subset $F\subset\Ga$ and $\ve>0$,  
we take $t_i$ and $C_i\subset D_i\subset\Ga$ for $1\le i \le n$ 
as in Definition.
Let $P_i \le Q_i$ denote the diagonal projections 
in $\IB(\ell_2\Ga)$ corresponding to the subsets $C_i\subset D_i$ and 
set $P_i' \coloneq t_i^{-1}Q_i^\perp t_i$ and 
$Q_i' \coloneq t_i^{-1}P_i^\perp t_i$.
Note that $\{ P_i \}\cup\{ P'_j\}$ are mutually orthogonal. 
Here $\Ga$ is represented in $\IB(\ell_2\Ga)$ 
by the left regular representation. 
We denote by $\omega$ the vector state on $\IB(\ell_2\Ga)$ 
corresponding to $\delta_1$, 
which satisfies $\omega|_{\mathrm{C}^*_\lambda(\Ga)}=\tau$. 
We define $T\in\IB(\ell_2\Ga)$ by 
\[
T\coloneq (2n)^{-1/2}\sum_i (P_i t_i + t_i^{-1} Q_i^\perp)
= (2n)^{-1/2}\sum_i (P_i t_i + P'_i t_i^{-1} ).
\]
The element $T$ is an almost isometry as
\[
T^*T
 = (2n)^{-1} \sum_i (t_i^{-1}P_i t_i  + t_i P_i' t_i^{-1})
 = 1-(2n)^{-1}\sum_i (Q_i'+Q_i)
 \approx_{\ve n^{-1/2}} 1.
\]
Moreover, $T^*aT = 0 = \tau(a)$ for every $a\in F\setminus\{1\}$. 
One has by Schur's test
\begin{align*}
\| \sum_i (Q_i-P_i)t_i \| 
 &= \| \sum_i (Q_i-P_i)t_i(Q_i'-P_i') \| 
 \le \|\sum_i Q_i\|^{1/2}\|\sum_i Q'_i\|^{1/2}
 \le\ve n^{1/2}
\end{align*}
and hence 
\[
\| T+T^* - (2n)^{-1/2}\sum_i (t_i+t_i^*) \|
 \le 2(2n)^{-1/2}\| \sum_i (Q_i-P_i)t_i \|
 \le 2\ve.
\]
Also, for every finite subset of $\Ga$, the corresponding 
diagonal projection $R\in\IB(\ell_2\Ga)$ satisfies 
$\| RT\|^2 \le (2n)^{-1} \|\sum_i P_i + P_i' \| \|\sum_i t_i^{-1}Rt_i + t_i R t_i^{-1}\| 
\le n^{-1}\rank R$. 
It follows that a suitable ultra-limit $\tilde{T} \in \IB(\ell_2\Ga)^\cU$ 
of $T$ as above verifies the condition 
of Theorem~\ref{thm:toeplitz} 
for $(\mathrm{C}^*_\lambda(\Ga),\tau)\subset(\IB(\ell_2\Ga),\omega)$ 
and hence $(\mathrm{C}^*_\lambda(\Ga),\tau)$ is completely selfless.
\end{proof}

\begin{prop}\label{prop:dlh}
Let $\Ga$ be a group. Assume that there
is a continuous action of $\Ga$ on a Hausdorff topological space 
that is minimal, extremely proximal, and topologically free. 
Then $\Ga$ has property $\mathrm{P}_{\mathrm{PHP}}$.
\end{prop}
\begin{proof}
We sketch the proof by adopting that 
of Lemma~4 in \cite{delaharpe}. 
Let a finite subset $F\subset\Ga$ 
and $\ve>0$ be given. Take $n \geq \ve^{-2}$. 
Then, there are $t_1,\ldots,t_n$ in $\Ga$ and 
mutually disjoint open subsets $U_i^\pm$ 
such that $t_i (L\setminus U_i^-)\subset U_i^+$ for every $i$. 
Moreover, we may assume that 
$a(U_i^+\cup U_i^-)\cap (U_j^+\cup U_j^-)=\emptyset$ 
for every $a\in F\setminus\{1\}$ and every $i,j$.
We fix a point $x_0\in L$ and set 
$C_i \coloneq D_i \coloneq \{ s\in \Ga : sx_0\in U_i^+\}$. 
They witness $\mathrm{P}_{\mathrm{PHP}}$.
\end{proof}

Acylindrically hyperbolic groups with trivial 
finite radical satisfy the assumption of 
Proposition~\ref{prop:dlh} (Proposition~0.3 in \cite{ad}). 
Hence it removes property RD assumption from 
Amrutam et al.'s theorem (\cite{agkep}), 
see also \cite{yang} for an alternative proof. 
However, unlike Amrutam et al.\ and Yang's 
theorem (\cite{agkep} \cite{yang}), 
it does not provide an element in $\Ga^\cU$ 
that is \emph{strongly free} from $\Ga$. 
The same comment applies to Theorem~\ref{thm:main} 
and the following proposition that 
partly generalizes Vigdorovich's theorem (\cite{vigdorovich}).

\begin{prop}
Every Zariski-dense subgroup $\Ga\le\mathrm{PSL}(d\geq2,\IR)$ 
satisfies property $\mathrm{P}_{\mathrm{PHP}}$.
\end{prop}
\begin{proof}
We follow the proof of Proposition~1 in \cite{bch}. 
Let $B:=\IP^{d-1}(\IR)$ denote the homogeneous space of 
$G \coloneq \mathrm{PSL}(d,\IR)$, 
on which $\Ga$ acts continuously. 
Take $t_0\in\Ga$ with eigenvalues 
$\lambda_1>\lambda_2 > \cdots >\lambda_n>0$ and 
denote by $[\xi_1 : \cdots : \xi_n]$ the corresponding 
homogeneous coordinates on $B$. 
We set $p \coloneq [1 : 0 : \cdots : 0] \in K\coloneq\{ [\xi_1 : \cdots : \xi_n] : \xi_n = 0\}$ and $p' \coloneq [0 : \cdots : 0 : 1] \in K'\coloneq\{ [\xi_1 : \cdots : \xi_n] : \xi_1 = 0\}$.
Then for every $x\in B\setminus\{p\}$, 
every limit point of $t_0^{-n} x$ is in $K'$; 
and for every $x\in B\setminus K$, 
one has $\lim_n t_0^{-n}x = p'$. 

Let a finite subset $F\subset\Ga$ and $\ve>0$ be given. 
We claim that for every $n$ 
there are $s_1,\ldots,s_n\in\Ga$ such that 
the members of 
\[
\{ a s_i p : a \in F, i=1,\ldots,n \}\cup \{ b s_j p' : b \in F, j=1,\ldots,n \}
\]
are mutually distinct and that the hyperplanes 
$\{ s_i K \}$ (resp.\ $\{ s_i K'\}$) are in general position,
i.e., for every $I\subset\{1,\ldots,n\}$  
the linear dimension of $\bigcap_{i\in I} s_iK$ 
is $\max\{d-|I|,0\}$ (resp.\ likewise).
The proof is by induction on $n$. 
Suppose that the claim holds true for $n-1$ 
and we are given $s_1,\ldots,s_{n-1}\in\Ga$ that meet the claim.
For each $a\in\Ga$ and $q,q'\in B$ (which may be the same), 
the subset 
$\{ s\in\Ga : sq\neq aq'\}$ is non-empty and Zariski-open. 
For each $a\in\Ga\setminus\{1\}$ 
and $q,q'\in B$, the subset 
$\{ s\in\Ga : sq\neq asq'\}$  
is also non-empty and Zariski-open. 
For example, suppose by contradiction 
that $sq=asq'$ for all $s\in\Ga$. 
Then $q=aq'$ and $sq=asq'=asa^{-1}q$ 
for all $s\in G$ 
by Zariski-density of $\Ga$ in $G$, 
but since $\bigcup_s s^{-1}G_q s = G$, 
where $G_q$ is the stabilizers of $q$, 
this implies that $a$ acts as identity on $B$, meaning that $a=1$. 
Also for each proper subspaces $L$ and $L'$ 
the subset $\{ s\in\Ga : L' \not\subset sL\}$ 
is non-empty and Zariski-open. 
Since a finite intersection of 
non-empty Zariski-open subsets is non-empty, 
there is $s_n$ such that $s_1,\ldots,s_n$ meet the claim. 

We continue the proof and let $s_1,\ldots,s_n$ 
be as in the previous claim. 
We choose neighborhoods 
$U_i$ of $s_ip$ and $V_i$ of $s_iK$ with $U_i\subset V_i$
and likewise $U_i'\subset V_i'$ that satisfy that 
\[
\{ a U_i : a \in F, i=1,\ldots,n \}\cup \{ b U_j' : b \in F, j=1,\ldots,n \}
\]
are mutually disjoint and that 
every $I\subset\{1,\ldots,n\}$ 
with cardinality $d$ satisfies 
$\bigcap_{i\in I} V_i = \emptyset$ and $\bigcap_{i\in I} V_i' = \emptyset$.
%
%
Then by compactness, for each $i$ 
there is $n_i\in\IN$ large enough so that 
$t_i\coloneq s_it_0^{n_i}s_i^{-1}$ satisfies 
$t_i^{-1}(B\setminus U_i)\subset V_i'$ and 
$t_i^{-1}(B\setminus V_i)\subset U_i'$. 
We further choose $x_0\in B$ and 
set $C_i\coloneq\{ g \in \Ga : gx_0 \in U_i \}$ 
and $D_i\coloneq\{ g\in\Ga : gx_0\in V_i\}$. 
Thus, for $n$ large enough, 
$t_1,\ldots,t_n$ and $C_i\subset D_i$ verify
$\mathrm{P}_{\mathrm{PHP}}$ with 

\begin{minipage}{.95\textwidth}
\[
\sup_{x\in\Ga} |\{ i : x\in D_i \} \cup \{ j : x \in t_j^{-1} (\Ga\setminus C_j)\}| \le 2(d-1) \le \ve n^{1/2}.
\]\end{minipage}
\end{proof}

\end{document}